\documentclass[12pt, reqno, a4paper]{amsart}
\usepackage{amsmath,mathrsfs}
\usepackage{amssymb}
\usepackage{color}
\usepackage{calligra}
\usepackage{dsfont}
\usepackage[latin1]{inputenc}

\newcommand\NoBlackBoxes{\global\overfullrule0pt}
\NoBlackBoxes

\newcommand{\N}{\mathbb{N}}

\makeatletter
\let\serieslogo@\relax
\let\@setcopyright\relax

\newtheorem{definition}{Definition}[section]
\newtheorem{theorem}[definition]{Theorem}
\newtheorem{lemma}[definition]{Lemma}

\newtheorem{rems}[definition]{Remarks}
\newtheorem{example}[definition]{Example}

\renewcommand{\P}{{\mathbb{P}}}
\newcommand{\E}{{\mathbb{E}}}
\newcommand{\R}{{\mathbb{R}}}

\renewcommand{\epsilon}{\varepsilon}
\renewcommand{\phi}{\varphi}

\numberwithin{equation}{section}

\begin{document}

\setcounter{page}{1}

\title[Large deviations and positive association]{Large deviation upper bounds for sums of positively associated indicators}

\author[Matthias L\"owe]{Matthias L\"owe}
\address[Matthias L\"owe]{Fachbereich Mathematik und Informatik,
University of M\"unster,
Einsteinstra\ss e 62,
48149 M\"unster,
Germany}

\email[Matthias L\"owe]{maloewe@math.uni-muenster.de}

\author[Franck Vermet]{Franck Vermet}
\address[Franck Vermet]{Laboratoire de Math\'ematiques, UMR CNRS 6205, Universit\'e de Bretagne Occidentale,  6, avenue Victor Le Gorgeu\\
CS 93837\\
F-29238 BREST Cedex 3\\
France}

\email[Franck Vermet]{Franck.Vermet@univ-brest.fr}


\date{\today}



\newcommand{\wlim}{\mathop{\hbox{\rm w-lim}}}
\newcommand{\na}{{\mathbb N}}
\newcommand{\re}{{\mathbb R}}

\newcommand{\vep}{\varepsilon}

\begin{abstract}
We give exponential upper bounds for $\P(S \le k)$, in particular $\P(S=0)$, where $S$ is a sum of indicator random variables that are positively associated. These bounds
allow, in particular, a comparison with the independent case. We give examples in which we compare with a famous exponential inequality for sums of correlated indicators, the Janson inequality. Here our bound sometimes proves to be superior to Janson's bound.
\end{abstract}

\maketitle

\section{Introduction}
Consider a family of indicator random variables $(X_i)_{i\in I}$ with
$$
\P(X_i=1)=p_i=1-\P(X_i=0).
$$
Let $X=\sum_{i\in I}X_i$ be their sum. Moreover, assume that the random variables are associated (or positively associated), i.e. for all coordinatewise increasing functions
$$
f,g: \R^{|I|} \to \R
$$
we have that
$$
\mathrm{Cov}(f(X_i, i\in I ), g(X_i, i\in I) ) \ge 0.
$$
Association has first been defined and analyzed in \cite{Esary/Proschan/Malkup:1967}.
Associated random variables occur in a variety of situations, e.g. in statistical mechanics in the context of the FKG inequalities, in subgraph count statistics in random graphs or in neural networks. For further application it is worth noting, of course, independent random variables are associated and that, if $Y_1, \ldots, Y_n$ are associated, and $f_i:\R^n \to \R$ is increasing for each $i=1, \ldots n$, then also $Z_1, \ldots Z_n$, with $Z_i = f_i(Y)$ are associated (see
Theorem 2.1 and (P4) in \cite{Esary/Proschan/Malkup:1967}). Here $Y=(Y_1, \ldots, Y_n)$.
Actually this fact allows to construct and/or discover positive association in many examples.
\begin{example}[k-Runs]\label{runs}
\normalfont
Let $Y_1, Y_2, \ldots Y_{n+k-1}$ be i.i.d. random variables with $\P(Y_i=1) =p =1-\P(Y_i=0)$ and for $k \in \N$ let
$$
Z_i = \prod_{j=i}^{i+k-1} Y_i
$$
be the indicator for a $k$-run of 1's starting in i. Then, as the $Z_i$ are increasing functions of the independent and thus associated random variables
$Y_1, \ldots, Y_n$, they are associated. $k$-runs are frequently used e.g. in sequence alignments in mathematical biology. We will be interested in the total
number $Z$ of $k$-runs in a Bernoulli sequence of a given length.
\end{example}

\begin{example}[Subgraph-counts]\normalfont \label{counts}
Let $(Y_{i,j})_{1\le i<j \le n}$ be the i.i.d. Bernoulli random variables occurring in the Erd\"os-R\'enyi random graph $G(n,p)$, i.e.
$Y_{i,j}=1$, if and only if the edge between $i$ and $j$ is present in the realization of the random graph, and this occurs with probability $p$. Otherwise $Y_{i,j}$ is 0. Now let $K_k$ be an arbitrary fixed complete graph, e.g. a triangle for which k=3.
The random variables
$$
Z_{s_1,s_2, \ldots s_k}=\prod_{1 \le i <j \le k} Y_{s_i, s_j}
$$
are associated. Here the $s_i$ are pairwise different and $\{s_1,s_2, \ldots s_k\} \subseteq \{1, \ldots, n\}$. The $Z_{s_1,s_2, \ldots s_k}$ indicate whether the realized graph contains a $K_k$ at the vertices $s_1, \ldots, s_k$. We will be interested in the number of $K_k$'s in a realization of $G(n,p)$, thus in the sum of the $Z_{s_1,s_2, \ldots s_k}$, i.e. $Z=\sum_{1 \le s_1 < \ldots < s_k \le n} Z_{s_1,\ldots, s_k}$ .
\end{example}

\begin{example}[U-statistics]\label{ustat}
\normalfont
Let $Y_1, Y_2, \ldots Y_{n}$ be i.i.d. random variables with $\P(Y_i=1) =p =1-\P(Y_i=0)$ and for $k \in \N$ and $1 \le i_1, \ldots, i_k \le n$ pairwise different let
$$
Z_{i_1,\ldots, i_k}  = \prod_{j=1}^{k} Y_{i_j}
$$
Then again the $Z_{i_1,\ldots, i_k}$ are associated. We will be interested in the U-statistics $Z=\sum_{1 \le i_1 < \ldots < i_k \le n} Z_{i_1,\ldots, i_k}$.
Note that the corresponding quantities in Examples \ref{runs} and \ref{counts} are structurally similar. However, there $Z$ are so-called incomplete U-statistics, i.e. not all possible random variables $Z_{i_1,\ldots, i_k}$ are actually considered (in Example \ref{counts}, e.g., one requires that the vertices form a clique, rather than just multiplying any collection of $k$ edges). This leads to ''more independence'' among the summands. We will see in Section 3 that the bound we give works particularly well for random variables with many but weak correlations.
\end{example}

\begin{example}[Random hypergraphs]\label{group}
\normalfont
The following example has several motivations. On the one hand, it is the most generalization of the famous random graph models $G(N,p)$ and $G(N,M)$ invented by Erd\"os and R\'enyi (\cite{ER59}). Here edges are realized between $N$ points i.i.d. with probability $p$ ($G(N,p)$), or we take one of the graphs with exactly $M$ edges at random with equal probability ($G(N,M)$). These models are readily generalized to hypergraphs in the following way. Given a set of hyperedges, e.g. all possible complete graphs with $k$ vertices that can be constructed on the vertices $\{1, \ldots, N\}$ we realize each of them independently at random with probability $p$ or we realize a hypergraph with exactly $M$ such hyperedges at random with equal probability. Such random hypergraphs have been studied e.g. in
\cite{C-OMS07}, \cite{BC-OK10}, \cite{BC-OK14}, or \cite{KL02}.

On the other hand,
 the example can be considered as a mean-field version of the very recent topic of loop-percolation (see e.g. \cite{romik}): We cover a graph with a random ensemble of subgraphs and ask for the properties of the thus emerging graph. Finally, our example is also motivated by the analysis of a certain neural network (see \cite{gripona}, \cite{griponb}). However, it may  be  regarded as an interesting combinatorial game with a certain similarity to group testing. Take the complete graph $K_N$ and in each of a finite number of steps delete the edges of a randomly chosen complete subgraph $K_k$ ($k$ fixed), if they are not already deleted. Let $Y_{i,j}$ denote the indicator of the event, that edge $\{i,j\}$ has been deleted after $n$ steps. Then the $(Y_{i,j})$ are associated (and so are the $(Z_{i,j})=(1-Y_{i,j})$). We will be interested in the probability that the random $K_k$'s cover all of $K_N$, hence that the random hypergraph we realize is connected.
\end{example}
Other examples for association have been found it statistical mechanics, e.g. in models obeying the FKG inequalities (\cite{newman_fkg}) or recently in the fuzzy fractional Potts model \cite{Kahn_weininger} or in the parabolic Anderson model \cite{karasik}.

Due to their frequent occurrence, deviation inequalities for positively associated random variables have been investigated in a number of recent articles. Already Newman and Wright prove an invariance principle \cite{newmanwright1} and martingale inequalities \cite{newmanwright2}. Cox and Grimmett show a Central Limit Theorem \cite{CoxGrimmett}. Boutsikas and Koutras present a simple
upper bound for the distance between the distribution of the sum of $n$ associated random variables $X_i$  and
a sum of $n$ independent random variables with the same marginals as $X_i$ \cite{BoutsikasKoutras}. Exponential inequalities (also partially together with a resulting Strong Law of Large Numbers) were derived by Oliviera \cite{Oliveira}, \cite{Sung}, \cite{Xing_Yang2}, and Yang and Chen \cite{YangChen}. Finally a Law of the Iterated Logarithm was proved by Xing and Yang \cite{Xing_Yang}.
A very readable survey over these and more results can be found in \cite{Oliveira2}.


The purpose of the present note is to give a new exponential bound for the probability that a sum of positively associated indicator variables is particularly small or even 0. These bounds will also allow for a comparison with the case of sums of independent indicators with the same distribution. Such a comparison is of interest, since often when positively associated random variables are applied in the literature, it is tacitly assumed that the probabilities of their sum being 0 is well approximated by the independent case (see e.g \cite{gripona}). Moreover, by positive association, we have that
$$
\P(\sum_i X_i =0) \ge \prod_i \P(X_i=0)
$$
(with the sum and the product being taken over the same set of indices), anyway.
Thus we always have an obvious bound in one direction by the independent case.
One way to obtain the comparison between the dependent and the independent situation is Janson's inequality (\cite{Janson1}, \cite{Janson2}). This can be stated as follows:
In the above setting, let $I$ be index set of the associated variables $(X_i)_{i \in I}$
\begin{eqnarray*}
X= \sum_{i \in I} X_i, \quad &&\quad \lambda= \E X\\
\Delta= \frac 12 \sum_{i \sim j,\atop i \neq j} \E(X_i X_j) && \quad \overline \Delta = \lambda + 2\Delta,
\end{eqnarray*}
where we write $i \sim j$ if and only if $\mathrm{Cov}(X_i, X_j) \neq 0$.
Then the following bounds hold.

\begin{theorem} \label{Jineq} (\cite{Janson1}, \cite{Janson2}, also see \cite{suen})
We have
\begin{equation}\label{one}
\P(X=0) \le e^{-\lambda + \Delta}
\end{equation}
and
\begin{equation}
\P(X=0) \le e^{-\frac{\lambda}{\overline{\Delta}^2}}.
\end{equation}
Moreover,
\begin{equation}\label{three}
\P(X=0) \le \exp\left(\frac{\Delta}{1-\max_{i\in I } E X_i}\right) \prod_{i\in I } (1-\E X_i).
\end{equation}
\end{theorem}
For the purpose of the present note, we will be most of all interested in comparing our results with \eqref{three}, which can be derived from \eqref{one}, but was first proved by Boppona and Spencer \cite{Boppona_Spencer}.

Another way to compare the dependent and the independent cases is the following result, by Boutsikas and Koutras.

\begin{theorem} \label{Boutsikas} (\cite{BoutsikasKoutras})
We have
\begin{equation}\label{four}
\P[X=0] \le \prod_{i=1}^n \P[X_i=0] + \sum_{i<j}\mathrm{Cov}(X_i,X_j).
\end{equation}
\end{theorem}

Our main theorem also gives an upper bound for the probability that the sum of positively or negatively associated indicator random variables is zero. It can be stated as follows.

\begin{theorem}\label{ours}
Let the $X_i$ be positively associated indicators.
With the above notation we have for all $t>0$
\begin{equation}\label{first}
\P(X=0) \le e^{-t |I|}\left(\prod_{i \in I}\E[ e^{t(1-X_i)}] +t^2 e^{t |I|} \sum_{i<j}\mathrm{Cov}(X_i, X_j)\right)
\end{equation}
If the $(X_i)$ are identically distributed with $\P(X_i=1)=p=1-\P(X_i=0)$ this boils down to
\begin{eqnarray} \label{second}
\P(X=0) &\le& e^{-t |I|} (p+(1-p)e^t)^{|I|} +t^2 \sum_{i<j}\mathrm{Cov}(X_i, X_j)\nonumber\\
& =&(1-p)^{|I|}(1+e^{-t}\frac p {1-p})^{|I|}+t^2\sum_{i<j}\mathrm{Cov}(X_i, X_j)
\end{eqnarray}
\end{theorem}

\begin{rems}
\normalfont
\begin{enumerate}
\item
It is obvious that Theorem \ref{Jineq}, Theorem \ref{Boutsikas}  as well as our Theorem \ref{ours} are tailor-made for the situation of weak correlations, where one is close to the independent situation. If correlations are getting large, then they give useless bounds:  The probabilities are bounded from above by numbers larger than 1.
An important difference between \eqref{four} and \eqref{first} is that we have a factor $t^2$ in front of the sum of the covariances. As we will see in Section 3, choosing $t$ small, this factor can be decisive when the covariances are too large.

\item
A quick glance at \eqref{three} and \eqref{second} already reveals the major difference between the two bounds: While \eqref{three} is purely multiplicative, our bound \eqref{second} also has an additive component. This has a consequence for the applicability of our bounds. Consider e.g. the situation for \eqref{second}:  In case that $\P(X=0)$  is comparable to the independent approximation $(1-p)^{|I|}$ and $p$ is much larger than $1/|I|$, both probabilities will typically go to 0 and they will do so at an exponential rate. To have the additive term $t^2\sum_{i<j}\mathrm{Cov}(X_i, X_j)$ in \eqref{second} smaller than the probability we want to approximate, we therefore need to choose $t$ extremely small, which in turn means that the multiplicative correction in $(1+e^{-t}\frac p {1-p})^{|I|}$ is of order $e^{p\,|I|}$ which typically is only effective, if $p$ is of order $1/|I|$. Hence in such cases, we expect that \eqref{three} will yield results that are more useful than \eqref{second}.
\end{enumerate}

\end{rems}

We will prove Theorem \ref{ours} in the following section. Section 3 will be devoted to comparing Theorems \ref{Jineq} and \ref{ours} in the examples mentioned above. We will see that our bound sometimes is better than Janson's inequality and also better than the bound given in Theorem \ref{Boutsikas}.

\section{Proof of Theorem \ref{ours}}
As usual the central ingredient of the proof of an exponential inequality as in Theorem \ref{ours} is an estimate of the moment generating function.
Such bound can be found in \cite{DewanRao99}, Lemma 3.1. 
 For the sake of completeness (and also as the original proof is somewhat short), we reprove the statement here.

\begin{lemma}(see \cite{DewanRao99}, Lemma 3.1)
Let $X_1, X_2, \ldots, X_n$ be positively associated random variables, such that $X_i \le \kappa$ for all $i$ and some $\kappa>0$.
Then for any $t>0$
$$
\left| \E \left[e^{t \sum_{i=1}^n X_i}\right]- \prod_{i=1}^n \E \left[ e^{tX_i}\right]\right| \le t^2 e^{n t\kappa} \sum_{1 \le i < j \le n} \mathrm{Cov}(X_i, X_j).
$$
\end{lemma}

\begin{proof}
The key tool is to use Lemma 3 in \cite{newman_fkg}. This Lemma implies that for associated random variables $X$ and $Y$ with finite variance, and (real or complex valued) functions $f$ and $g$ with bounded derivatives one has
$$
|\mathrm{Cov}(f(X),g(Y)| \le ||f'|| \, ||g'|| \, \mathrm{Cov}(X,Y)
$$
where $||\cdot||$ is the sup-norm. Applying this we obtain due to the boundedness of the the $X_i$
$$
|\mathrm{Cov}(e^{tX_1},e^{tX_2})| \le t^2 e^{2 t \kappa} \, \mathrm{Cov}(X_1,X_2)
$$
which is precisely the assertion for $n=2$. The result now follows by induction:
According to our earlier remark, if $X_1, X_2, \ldots, X_n$ are associated random variables, then so are $X_1 + X_2 + \ldots, X_{n-1}$ and $X_n$. Thus , if the assertion has already been proven
for $n-1$ we obtain
\begin{eqnarray}
&& \left| \E \left[e^{t \sum_{i=1}^n X_i}\right]- \prod_{i=1}^n \E \left[ e^{tX_i}\right]\right| \nonumber \\&\le&
\left| \E \left[e^{t \sum_{i=1}^n X_i}\right]- \E \left[e^{t \sum_{i=1}^{n-1} X_i} \right]\E \left[e^{tX_n} \right]\right| +
\left|
\E \left[e^{t \sum_{i=1}^{n-1} X_i} \right]\E \left[e^{tX_n} \right]-  \prod_{i=1}^n \E \left[ e^{tX_i}\right]\right| \nonumber \\
&\le& t^2 e^{n t\kappa}\sum_{1 \le i  \le n-1} \mathrm{Cov}(X_i, X_n) + t^2 e^{n t\kappa} \sum_{1 \le i < j \le n-1} \mathrm{Cov}(X_i, X_j)\label{induction}\\
&\le & t^2 e^{n t\kappa} \sum_{1 \le i < j \le n} \mathrm{Cov}(X_i, X_j),
\end{eqnarray}
where in \eqref{induction} the first summand is bounded by another application of Lemma 3 in \cite{newman_fkg} inequality, while the bound on the second summand follows from the induction hypothesis.
\end{proof}

From this bound the proof of the main theorem follows quickly:

\begin{proof}[Proof of Theorem \ref{ours}]
With the notation of the theorem let $Y_i= 1-X_i$. Note that together with the $(X_i)$ also the $(Y_i)$ are associated.
Using an exponential Markov inequality and the previous lemma we obtain for all $t>0$
\begin{eqnarray*}
\P(X=0) &=& \P(\sum_{i \in I} X_i =0) = \P(\sum_{i \in I} Y_i =|I|)= \P(\sum_{i \in I} Y_i \ge|I|)\\
&\le & e^{-t |I|} \E[e^{t\sum_{i \in I} Y_i}]\\
 &\le & e^{-t |I|} \left(\prod_{i \in I}\E[ e^{tY_i}] + t^2 e^{t |I|} \sum_{i < j } \mathrm{Cov}(Y_i, Y_j)\right)\\
 &=& e^{-t |I|} \left(\prod_{i \in I} \E[e^{t(1-X_i)}] + t^2 e^{t |I|} \sum_{i < j} \mathrm{Cov}(X_i, X_j)\right)\\
\end{eqnarray*}
which is \eqref{first}.

Computing the expectation in the product for identically distributed random indicators with success probability $p$ we arrive at
\begin{eqnarray*}
\P(X=0) &\le& e^{-t |I|} \left(\prod_{i \in I} (p+(1-p)e^t) + t^2 e^{t |I|} \sum_{i < j} \mathrm{Cov}(X_i, X_j)\right)\\
        &=&  e^{-t |I|} (p+(1-p)e^t)^{|I|} +  t^2 \sum_{i < j} \mathrm{Cov}(X_i, X_j) \\
        &=& (1-p)^{|I|} (1+ \frac{p}{1-p} e^{-t})^{|I|}+  t^2 \sum_{i < j} \mathrm{Cov}(X_i, X_j)
\end{eqnarray*}
which is \eqref{second}.
\end{proof}

\section{Examples}
In this section we will discuss the quality of our estimate in the examples given in Section 1. The aim is to see, if and when Theorem \ref{ours} is better than Theorem \ref{Jineq}, especially, when \eqref{second} is more powerful than \eqref{three}, and to gain a general understanding, in which situations it is advisable to apply Theorem \ref{ours}.

We start with
\begin{example}[k-Runs, Example \ref{runs} continued]
\normalfont
To keep the situation symmetric and simplify computations, we slightly change the setup. So let $Y_1, Y_2, \ldots Y_{n}$ be i.i.d. random variables with $\P(Y_i=1) =p =1-\P(Y_i=0)$ and for $k \in \N$ let
$$
Z_i = \prod_{j=i}^{i+k-1} Y_i
$$
(where $i+k-1$ now is to be taken modulo $n$).
Moreover let $Z=\sum_{i=1}^n Z_i$. Computing the central quantity in Janson's inequality we obtain
\begin{eqnarray*}
\Delta &= & 
\frac n 2  \sum_{1 \sim j, \atop 1 \neq j} \E Z_1 Z_j = \frac n 2 \sum_{j=1}^{k-1} p^{k+j}\\
&=& \frac n2 p^{k+1} \frac {1-p^{k-1}}{1-p}
\end{eqnarray*}
Thus \eqref{three} yields
\begin{equation}\label{Jansonruns}
\P(Z=0) \le (1-p^k)^n \exp\left(\frac {n p^{k+1}( 1-p^{k-1})}{2 (1-p)(1-p^k)}\right).
\end{equation}
(Note that more generally the total variation distance between the distribution of $Z$ and a Poisson distribution was bounded in \cite{Barbour_Holst_Janson}, Theorem 8.F.)
This bound becomes effective, i.e. we have that $\P(Z=0)$ is well approximated by $(1-p^k)^n$, if $p \ll n^{-1/k+1}$, i.e. if $np^{k+1} \to 0$.
On the other hand our Theorem \ref{ours} yields for any $t>0$
\begin{eqnarray*}
\P(Z=0) & \le & (1-p^k)^n (1+e^{-t}\frac{p^k}{1-p^k})^n+t^2\sum_{i<j}\mathrm{Cov}(Z_i, Z_j)\\
& \le & (1-p^k)^n (1+e^{-t}\frac{p^k}{1-p^k})^n+ t^2 \Delta
\end{eqnarray*}
Choosing $0<t \ll \Delta^{-1/2}$ the second summand on the right hand side converges to 0. Then the correction factor $(1+e^{-t}\frac{p^k}{1-p^k})^n$ for the first summand on the right hand side is bounded by $e^{n \frac{p^k}{1-p^k}}$ and particularly converges to 1, if $p \ll n^{-1/k}$, which improves the bound from Janson's inequality.
Note however that even for much larger $p$ an exponential bound on $\P(Z=0)$ holds true. As a matter of fact Barbour et al. in \cite{Barbour_Holst_Janson},
Theorem 8.G  give the bound
\begin{equation}\label{Janson_Poisson_runs}
|\P(Z=0)- \exp(-n (1-p)p^k)| \le (2k(1-p)+1)p^k
\end{equation}
\end{example}

The recipe seen in the previous example illustrates the general method. In \eqref{second} one always chooses $t$ (depending on $n$) so small that the second summand on the right vanishes, and then we obtain a bound of the form $(1-p)^n e^{n \frac{p}{1-p}}$, where $p$ is the success probability of the random variables in question. This works, if $p$ is so small, that $(1-p)^n$ is of constant order. However, even if both $(1-p)^n$ and $\P(Z=0)$ tend to 0, we may choose $t$ so small, that the second summand in \eqref{second} vanishes, but we may not be able to guarantee that then $p$ is so small, that $e^{n \frac{p}{1-p}}$ converges to 1.

\begin{example}[Subgraph-counts, Example \ref{counts} continued]\normalfont
In the situation of Example \ref{counts} consider the number of triangles in the random graph $G(n,p)$.
Thus we consider
$$
Z_{s_1,s_2,s_3}=Y_{s_1, s_2}Y_{s_2, s_3} Y_{s_3, s_1}
$$
and
$$Z= \sum_{1 \le s_1 < s_2 < s_3 \le n} Z_{s_1,s_2,s_3}. $$
We bound the probability that the random graph does not contain any triangle, hence that $Z=0$.
The quantity $\Delta$ in \eqref{three} can this time be computed as
$$
\Delta = \frac 12 \binom n 3 3n p^5 \sim \frac 14 n^4 p^5,
$$
as there are $\binom n 3$ triangles, each of which is correlated with $3n$ other triangles, and in this case $\E Z_{s_1,s_2,s_3} Z_{s_4,s_5,s_6} =p^5.$
Thus Janson's inequality gives
$$
\P(Z=0) \le (1-p^3)^{\binom n3} \exp\left(\frac {n^4p^5}{4(1-p^3)}\right).
$$
If $p \ll n^{-4/5}$, this bound is effective, since the factor is tending to 1.

This time our bound works well only in some range of the parameters and for slightly different reasons than in the previous example. From \eqref{second} we obtain for any $t>0$
$$
\P(Z=0) \le (1-p^3)^{\binom n3}\left(1+e^{-t}\frac{p^3}{1-p^3}\right)^{\binom n3} +t^2 n^4 p^5.
$$

In the regime, where $n^{-1} \ll p $ we see exactly the problem mentioned at the end of the introduction occurring: The probabilities to be approximated are of order $\exp(-p^3 n^3)$ (at least for $p$ going to 0 with $n$ to infinity) and thus converging to 0 exponentially fast. So even if $p \ll n^{-4/5}$, in which case the second term on the right goes to 0, we need to choose $t$ to go to 0 as well, and we cannot improve the first summand in our bound by choosing
$t$ large. If now $n^{-1} \ll p \ll n^{-4/5}$ we may obtain a bound of the form $(1-p^3)^{\binom n3}\left(1+\frac{p^3}{1-p^3}\right)^{\binom n3}$ by choosing $t$ extremely small. This is however useless, as the second factor on the right diverges.

So let us assume that $p=\lambda/n$. Choose $t=n^\alpha$ for $0<\alpha<\frac 12$. Then \eqref{second} becomes
\begin{eqnarray*}
\P(Z=0) & \le & (1-p^3)^{\binom n3}\left(1+e^{-n^\alpha}\frac{p^3}{1-p^3}\right)^{\binom n3} + \frac{ \lambda^5}{n^{1-2\alpha}}\\
&=& (1-p^3)^{\binom n3}\left(1+e^{-n^\alpha}\frac{p^3}{1-p^3}\right)^{\binom n3} + o(1).
\end{eqnarray*}
The multiplicative correction $\left(1+e^{-n^\alpha}\frac{p^3}{1-p^3}\right)^{\binom n3} \le \exp(\lambda^3 e^{-n^\alpha})$ converges to 1 much faster than $\exp(\frac 1n)$, the correction from
\eqref{three}.

\end{example}

\begin{example}[U-statistics, Example \ref{ustat} continued]
\normalfont
In the situation of Example \ref{ustat}, we define
 $$\displaystyle Z=\sum_{i=\{i_1,\ldots, i_k\}\subset \{1,\ldots,n\}} Z_i.$$
 We first compute the central quantity in \eqref{three}:
$$
\Delta = \frac 12 \binom n k \sum_{j=1}^{k-1} \binom{n-k}{j} p^{k+j} \sim C \frac 1 {k!} (np)^{k+1} \frac{1-(np)^{k-1}}{1-np},
$$ with $C\in ]0,1[$. Thus \eqref{three} yields
\begin{equation}\label{janson-ustat}
\P(Z=0) \le (1-p^k)^{\binom n k}\exp\left(C \frac{(np)^{k+1} (1-(np)^{k-1})}{k! (1-np)(1-p^k)}\right)
\end{equation}
for some constant $C$, and it is obvious that this bound is only effective, if $p=\frac{\lambda} n$, for some $\lambda >0$ or even $p= o(1/n)$. This is also the situation where Theorem \ref{ours} gives good bounds.
If $p = \lambda/n$, we choose $t \ll e^{-\lambda^k/2 k!}/\Delta$ to guarantee that the second summand in \eqref{second} is much smaller than the first. Then, for large $\lambda$, \eqref{second} basically boils down to
$$
\P(Z=0)\le  (1-p^k)^{\binom n k} \left(1+\frac {p^k}{1-p^k}\right)^{\binom nk}+o(1).
$$
This is still better than the bound \eqref{janson-ustat} obtained from \eqref{three}, since our multiplicative error factor behaves like $e^{\lambda^k/k!}$, while the one in \eqref{janson-ustat} is essentially $e^{\lambda^{2k}/k!}$.  However, if $\lambda$ is getting small, our bound gets even better (compared to \eqref{janson-ustat}), due to the additional factor $e^{-t}$ (with $t$ chosen as above) in
the multiplicative error term. This especially applies when $\lambda \to 0$ with $n$ going to infinity.
\end{example}

In the previous examples we always estimated the sum of the covariances occurring in \eqref{second} by the $\Delta$ from Janson's inequality: One may wonder whether we haven't been giving away too much by such a bound. The answer is no, since in these cases the covariances of the variables are dominated by the expectations of the product. However, this need not be the case in general and it will turn out in such situations our bound is significantly better than Janson's inequality.

\begin{example}[Random hypergraphs, Example \ref{group} continued]
\normalfont
Here we will see, that our inequality is able to outperform Janson's inequality in certain situations.

To this end, in the situation of Example \ref{group} let $\mathfrak{Y}$ be the event that all edges of $K_N$ are covered by one of the randomly chosen $K_k$'s. With the notation of Example \ref{group}
$$
\mathfrak{Y}= \left\{\sum_{1 \le i<j\le N} Y_{i,j}= \binom N2 \right\} = \left\{Z:= \sum_{1 \le i<j\le N} Z_{i,j}=0\right\}.
$$
We want to estimate the probability of $\mathfrak Y$ for large $n$ and fixed $k$. To this end assume we select $n$ $K_k$'s as subgraphs of $K_N$ i.i.d. at random. Then for any edge $(i,j)$
$$
\P(Z_{i,j}=1)  = \left(1-\frac{\binom{N-2}{k-2}}{\binom N k}\right)^n.
$$
Indeed, $Z_{i,j}=1$, if and only if $(i,j)$ is not contained in any of the $n$ independent $K_k$'s and
 this happens with probability  $\left(1-\frac{\binom{N-2}{k-2}}{\binom N k}\right)^n=:p(N,n,k)=:p$.
For $N$ large, and $k$ fixed,
$$
p \sim \left( 1-\frac{k(k-1)}{N^2}\right)^n.
$$
To compute $\Delta$ from Theorem \ref{Jineq} observe that indeed none of the $Z_{i,j}$ are uncorrelated, but that their correlation is extremely weak. As a matter of fact, there are two different types of correlations between
$Z_{i,j}$ and $Z_{i',j'}$ depending on whether the edges $(i,j)$ and $(i',j')$ share a vertex or not. Assume $(i,j)$ and $(i',j')$ share a vertex, e.g. $j=j'$. Then
with $Y_{i,j}^r$ denoting the indicator that $(i,j)$ is covered by the $r$'th independent copy of $K_k$
\begin{eqnarray*}
\P(Z_{i,j}Z_{i',j}=1) &=& \P(Y_{i,j}+Y_{i',j}=0)\\
&=& \P\left(\bigcap_{r=1}^n \{Y_{i,j}^r+Y_{i',j}^r=0\}\right)\\
&=& \left(\frac{\binom{N-3}{k}}{\binom N k }+3\frac{\binom{N-3}{k-1}}{\binom N k }+ \frac{\binom{N-3}{k-2}}{\binom N k }\right)^n,
\end{eqnarray*}
the different terms coming from the five possible cases:

\begin{itemize}
\item choose $k$ indices in $\{1,\ldots, N\}\setminus \{i,i',j\}$,
\item choose $i$ and then $k-1$ indices in $\{1,\ldots, N\}\setminus  \{i,i',j\}$,
\item choose $i'$ and then $k-1$ indices in $\{1,\ldots, N\}\setminus  \{i,i',j\}$,
\item choose $j$ and then $k-1$ indices in $\{1,\ldots, N\}\setminus  \{i,i',j\}$,
\item choose $i$ and $i'$ and then $k-2$ indices in $\{1,\ldots, N\}\setminus  \{i,i',j\}$.
\end{itemize}

In the same way, if $(i,j)$ and $(i',j')$ do not share a vertex, we have
\begin{eqnarray*}
\P(Z_{i,j}Z_{i',j'}=1) &=& \P\left(\bigcap_{r=1}^n \{Y_{i,j}^r+Y_{i',j'}^r=0\}\right)\\
&=& \left(\frac{\binom{N-4}{k}}{\binom N k }+4\frac{\binom{N-4}{k-1}}{\binom N k }+ 4\frac{\binom{N-4}{k-2}}{\binom N k }\right)^n
\end{eqnarray*}
Thus
\begin{eqnarray*}
\Delta &= &\frac 12\sum_{(i,j),(i',j'),\atop (i,j)\neq(i',j')} \E(Z_{i,j}Z_{i',j'})\\
&=& \binom N2 \left((N-2)  \left(\frac{\binom{N-3}{k}}{\binom N k }+3\frac{\binom{N-3}{k-1}}{\binom N k }+ \frac{\binom{N-3}{k-2}}{\binom N k }\right)^n \right. \\
&& \qquad \left.
 +\frac 12\binom{N-2}2   \left(\frac{\binom{N-4}{k}}{\binom N k }+4\frac{\binom{N-4}{k-1}}{\binom N k }+ 4\frac{\binom{N-4}{k-2}}{\binom N k }\right)^n\right)
\end{eqnarray*}
Now suppose $n=\lambda N^2$ for some $\lambda>0$ and that $N$ is large. Observe that for $n$ smaller than this, we cannot hope for any kind of Poisson approximation for the probability in question, since we simply do not have enough edges from the $K_k$'s to cover $K_N$. For $n$ as large as  $n=\lambda N^2$, or larger, the Poisson approximation becomes more likely, since then the probability not to cover an edge, becomes small.

We will now expand the exponentials in the computation of $\Delta$. Let us start with $k=3$, such that only the first of the two summands for $\Delta$ is present. We obtain:
\begin{eqnarray*}
\frac{\binom{N-3}{k}}{\binom N k }&=& 
(1- \frac kN)(1-\frac{k+1}N) (1-\frac{k+2}N)\frac N{N-1}\frac N{N-2}\\
&=&1- 3 \frac kN +\frac 1{N^2}(3k^2-3k)+ O(\frac 1{N^3})) \\
\end{eqnarray*}
and

\begin{eqnarray*}
\frac{\binom{N-3}{k-1}}{\binom N k }&=& 
\frac{(N-k)(N-k-1)k}{N(N-1)(N-2)}\\
&=&(1- \frac kN)(1-\frac{k+1}N) \frac{k}N(1+\frac 1N+ O(\frac 1{N^2}))(1+\frac 2N+ O(\frac 1{N^2}))\\
&=& \frac{k}N(1-2\frac {k-1}N+ O(\frac 1{N^2}))\\
\end{eqnarray*}

as well as

$$\frac{\binom{N-3}{k-2}}{\binom N k }= \frac{ (N-3)!}{N!} \frac{ (N-k)!}{(N-k-1)!}  \frac{ k!}{(k-2)!} = \frac{(N-k)k(k-1)}{N(N-1)(N-2)} = \frac{k(k-1)}{N^2}+ O(\frac 1{N^3}).$$

Thus
\begin{eqnarray*}
\P(Z_{i,j}Z_{i',j}=1)
&=& \left(\frac{\binom{N-3}{k}}{\binom N k }+3\frac{\binom{N-3}{k-1}}{\binom N k }+ \frac{\binom{N-3}{k-2}}{\binom N k }\right)^n\\
&=& (1- \frac 2{N^2} k(k-1) +O(\frac 1{N^3}))^n\\
\end{eqnarray*}

As for $n= \lambda N^2$ the probability $p$ is of constant order
we obtain from Janson's inequality \eqref{three} for $k=3$
\begin{eqnarray*}
\P(Z=0) &\le& (1-p)^{\binom N2}\exp\left(C N^3 (1- \frac 2{N^2} k(k-1) +O(\frac 1{N^3}))^n\right)\\
&\le &(1-p)^{\binom N2}\exp\left(C N^3 e^{-12\lambda}\right)
\end{eqnarray*}
for $N$ large enough. This is, of course, a useless bound, since the right hand side is converging to infinity.

On the other hand, for $\mathrm{Cov}(Z_{i,j}, Z_{i',j})$ we compute up to terms of order $n/N^4$
$$
\mathrm{Cov}(Z_{i,j}, Z_{i',j})=\frac{n(k^3 - 3k^2 + 2k)}{N^3} + O(\frac n {N^4}).
$$
Thus
$$
\sum_{(i,j) \sim (i',j')}\mathrm{Cov}(Z_{i,j}, Z_{i',j'})= C \lambda N^2.
$$

Obviously, the bound \eqref{four} from Theorem \ref{Boutsikas} by Boutsikas and Koutras is useless in this example, since the sum of the covariances diverges.
With the factor $t^2$ in front of the covariances, our bound \eqref{second} gives for any $t>0$.
$$
\P(Z=0) \le (1-p)^{\binom N2}(1+e^{-t}\frac p {1-p})^{\binom N2}+C t^2 \lambda N^2.
$$
If we now choose $t$ extremely small, e.g $t=e^{-N^3}$, it is evident that the second summand on the right is asymptotically negligible with respect to the first. This first summand is estimated as
$$
(1-p)^{\binom N2}(1+e^{-t}\frac p {1-p})^{\binom N2}\le (1-p)^{\binom N2}(1+\frac {e^{-6\lambda}} {1-e^{-6\lambda}})^{\binom N2}
$$
Thus \eqref{second} gives
\begin{equation*}
\P(Z=0)\le \left((1-p)(1+\frac {e^{-6\lambda}} {1-e^{-6\lambda}})\right)^{\binom N2}
\end{equation*}
In particular for $\lambda$ large, this is not only of the right order, but also very close to the lower bound $(1-p)^{\binom N2 }$.

For $k \ge 4$ one can compute that the covariances are now of order $1/N^2$ (which we spare ourselves as the computations are similar to ones above) , but as a matter of fact, a similar choice of $t$ leads to a similar result as in the case $k=3$.
\end{example}

To conclude, the above examples show that for positive correlations our method gives comparable results to that of Janson's inequality.
Even more is true: There are examples where our methods gives useful bounds, while \eqref{three} does not. These examples occur, if many of the random variables are correlated, but correlation is very weak. In this case $\Delta$ can tend to be so large, that \eqref{three} does not yield reasonable bounds. Moreover, when the sum of the covariances is too large, the last example illustrates also the situation where our bound gives an interesting approximation of the independent case, while the bound by Boutsikas and Koutras doesn't apply.

\bibliographystyle{abbrv}

\bibliography{LiteraturDatenbank}
\end{document}